\documentclass[12pt]{article}

\usepackage{amssymb,xcolor,graphicx,amsthm}
\usepackage{amsmath}

\newtheorem{theorem}{Theorem}[section]
\newtheorem{lemma}[theorem]{Lemma}
\newtheorem{proposition}[theorem]{Proposition}
\newtheorem{corollary}[theorem]{Corollary}

\newcommand{\ben}{\begin{enumerate}}
	\newcommand{\een}{\end{enumerate}}

\newcommand{\bt}{\begin{theorem}}
	\newcommand{\et}{\end{theorem}}
\newcommand{\bl}{\begin{lemma}}
	\newcommand{\el}{\end{lemma}}
\newcommand{\bc}{\begin{corollary}}
	\newcommand{\ec}{\end{corollary}}
\newcommand{\bp}{\begin{proposition}}
	\newcommand{\ep}{\end{proposition}}
\newcommand{\br}{\begin{remark}}
	\newcommand{\er}{\end{remark}}

\newcommand{\am}[1]{{\color{black} #1}}
\newcommand{\jm}[1]{{\color{black} #1}}

\newtheorem{remark}{Remark}
\newtheorem{example}{Example}

\providecommand{\keywords}[1]
{
	\small	
	{\noindent\textit{Keywords: }} #1
}

\providecommand{\ams}[1]
{
	\small	
	{\noindent\textit{2020 Mathematics Subject Classification: }} #1
}

\usepackage{amsmath}

\title{How many digits are needed?}

\author{I.~W.~Herbst, J. M\o ller, A.~M.~Svane}
\date{}

\begin{document}

\maketitle

\begin{abstract}
Let $X_1,X_2,...$ be the digits in the base-$q$ expansion of a random variable $X$ defined on $[0,1)$ where $q\ge2$ is an integer. For $n=1,2,...$, we study the probability distribution $P_n$ of the (scaled) remainder $T^n(X)=\sum_{k=n+1}^\infty X_k q^{n-k}$: If $X$ has an absolutely continuous CDF then $P_n$ converges in the total variation metric to the Lebesgue measure $\mu$ on the unit interval. Under weak smoothness conditions we establish first a coupling between $X$ and a non-negative integer valued random variable $N$ so that $T^N(X)$ follows $\mu$ and is independent of $(X_1,...,X_N)$, 
and second exponentially fast convergence of $P_n$ and its PDF $f_n$. We discuss how many digits are needed and show examples of our results. The convergence results are extended to the case of a multivariate random variable defined on a unit cube.
\end{abstract}


\ams{60F25; 62E17; 37A50}

\section{Introduction}
\label{s:intro}

Let $X$ be a random variable so that $0\le X<1$, and for $x\in\mathbb R$, let $F(x)=\mathrm P(X\le x)$ be the cumulative distribution function (CDF) of $X$.
For a given integer $q\ge2$, we consider the base-$q$ transformation $T:[0,1)\mapsto[0,1)$ given by 
\begin{equation}\label{e:Tq}T(x)=xq-\lfloor xq\rfloor
\end{equation}
where $\lfloor\cdot\rfloor$ is the floor function (so $\lfloor xq\rfloor$ is the integer part of $xq$). For $n=1,2,...$, let $T^n=T\circ\cdots\circ T$ denote the composition of $T$ with itself $n$ times and define 
\begin{equation}X_n=\lfloor T^{n-1}(X)q\rfloor
\end{equation}
where $T^{0}(X)=X$. Then 
\begin{equation}\label{e:q}
X=\sum_{n=1}^\infty X_nq^{-n}
\end{equation}
is the base-$q$ expansion of $X$ with digits $X_1,X_2,...$. Note that $X$ is in a one-to-one correspondence to 
the first $n$ digits $(X_1,...,X_n)$ together with $T^{n}(X)=\sum_{k=n+1}^\infty X_kq^{n-k}$, which is the remainder multiplied by $q^n$.
Let $\mu$ denote Lebesgue measure on $[0,1)$, $P_n$ the probability distribution of $T^n(X)$ and $F_n$ its CDF, so $X$ follows $P_0$ and has CDF $F_0=F$.
The following facts are well-known (see \cite{part1}): 
\begin{enumerate}
	\item[(a)] $P_0=P_1$ (i.e., invariance in distribution under $T$) is equivalent to stationarity of the process $X_1,X_2,...$.
	\item[(b)]  $P_0=P_1$ and $F$ is absolutely continuous if and only if $P_0=\mu$.
	\item[(c)] $P_0=\mu$ if and only if $X_1,X_2,...$ are independent and uniformly distributed on $\{0,1,...,q-1\}$.
\end{enumerate}
Items (a)--(c) 
together with the fact that $T$ is ergodic with respect to $\mu$ are used in metric number theory 
(see \cite{Karma}, \cite{Fritz}, and the references therein)
to establish properties such as 
`for Lebesgue almost all numbers between 0 and 1, the relative frequency of any finite combination of digits of a given length $n$ and which occurs among the first $m > n$ digits converges to $q^{-n}$ as $m\rightarrow\infty$'  
(which is basically the definition of a normal number in base-$q$, cf.\ \cite{Borel}).
To the best of our knowledge, less (or perhaps no) attention has been paid to the asymptotic behaviour of the (scaled) remainder $T^n(X)$ as $n\rightarrow\infty$. 
This paper fills this gap. 


Assuming $F$ is absolutely continuous with a probability density function $f$ we establish the following.
We start in Section~\ref{s:pre} to consider a special case of $f$ where $T^n(X)$ follows exactly $\mu$ when $n$ is sufficiently large. Then 
in Section~\ref{s:couplings}, under a weak assumption on $f$, we specify an interesting coupling construction involving a non-negative integer-valued random variable $N$ so that $T^N(X)$ follows exactly $\mu$ and is independent of $(X_1,...,X_N)$. 
 Moreover, 
in Section~\ref{s:results-q}, 
we show that $\lim_{n\rightarrow\infty}d_{\mathrm{TV}}(P_n,\mu)=0$ where $d_{\mathrm{TV}}$ is the total variation metric (as given later in \eqref{e:TV2}).
Because of these results, if in an experiment a realization of $X$
is observed and the first $n$ digits are kept, and if (so far) the only model assumption
is absolute continuity of $F$, then the remainder rescaled by $q^n$ is at least approximately
uniformly distributed when $n$ is large. Since we interpret the uniform distribution
as the case of complete randomness, no essential information about the distribution is lost. 
On the other hand, if the distribution of the remainder is far from uniform, this may indicate that the distribution one is trying to find has finer structure that one is missing by looking only at the first $n$ digits.  We return to this issue in Section~\ref{s:examples} when discussing sufficiency and ancillarity.
Furthermore, in Section~\ref{s:results-q} we study the convergence rate of $d_{\mathrm{TV}}(P_n,\mu)$ and other related properties.  In Section~\ref{s:examples}, we
	illustrate our results from Sections~\ref{s:couplings} and \ref{s:results-q} in connection to various specific choices of $F$, including the case where $F$ follows the extended Newcomb-Benford law (Example~\ref{ex:1}).
Finally, in Section~\ref{s:multi}, we generalize our convergence results to the situation 
where $X$ is extended to a multivariate random variable with values in the $k$-dimensional unit cube $[0,1)^k$ and each of the $k$ coordinates of $X$ is transformed by $T$. 

We plan in a future paper to study the asymptotic behaviour of the remainder in other expansions, including
a certain base-$\beta$ expansion of a random variable, namely when $q$ is replaced by $\beta=(1+\sqrt 5)/2$ (the golden ratio) in all places above.

\section{Preliminaries}\label{s:pre}

Let again the situation be as in \eqref{e:Tq}--\eqref{e:q}. The following lemma is true in general (i.e., without assuming $F$ is absolutely continuous). As in \cite{part1}, we define  a  base-$q$ fraction  in $[0,1)$ to be a number of the form $\sum_{k=1}^n j_kq^{-k}$ with $(j_1,...,j_n)\in\mathbb \{0,1,...,q-1\}^n$ and  $n\in\mathbb N$. 

\bl\label{l:1} If $F$ has no jump at any base-$q$ fraction in $[0,1)$ then for every $x\in[0,1]$,
	\begin{equation}\label{e:Fn}
	F_n(x) = 
	\sum_{j=0}^{q^n - 1} F(q^{-n}(j+x)) - F(q^{-n}j).
	\end{equation}
\el

\begin{proof}
	Clearly, \eqref{e:Fn} holds for $x=1$, so 
	let $0\le x<1$. For $j_1,...,j_n\in\{0,1,...,q-1\}$ and $j=\sum_{i=1}^n j_iq^{n-i}$, the event that $X_1=j_1, ..., X_n=j_n$, and $T^n(X)\le x$ is the same as the event that $q^{-n}j\le X<q^{-n}(j+1)$ and $X\le q^{-n}(j+x)$. Hence, since $0\le x<1$,
	\[F_n(x)=\sum_{j=0}^{q^n-1}\mathrm P(q^{-n}j\le X\le q^{-n}(j+x))\]
	whereby \eqref{e:Fn} follows since $F(x)$ has no jumps at the base-$q$ fractions. 
\end{proof}

The property that $F$ has no jump at any base-$q$ fraction 
is of course satisfied when $F$ is continuous. 

For the remainder of this 
section and the following Sections~\ref{s:couplings}--\ref{s:examples}
we
assume that $X$ has a probability density function (PDF) $f$ concentrated on $(0,1)$, meaning that $F$ is absolutely continuous with $F(x)=\int_{-\infty}^x f(t)\,\mathrm dt$ for all $x\in\mathbb R$.  
Then, by 
\eqref{e:Fn}, $F_n$ is absolutely continuous with PDF
\begin{equation}\label{e:5555}
f_n(x)=q^{-n}\sum_{j=0}^{q^n - 1} f(q^{-n}(j+x))
\end{equation}
for $0< x< 1$.

In the following special case of $f$, 
convergence of 
$P_n$
is obtained within a finite number of steps.

\bp\label{p:1}
Let $m\ge1$ be
an integer. Then
$P_m=\mu$ (and hence $P_n=\mu$ for $n=m,m+1,...$) if and only if for all $k\in\{0,1,...,q^m-1\}$ and Lebesgue almost every $u\in[0,1)$, 
\begin{equation}\label{e:fkuqm}
f((k+u)q^{-m})=q^m\mathrm P\left(\sum_{i=1}^m X_iq^{m-i}=k\,\bigg|\,T^m(X)=u\right).
\end{equation}
In particular, if
$f$ is constant Lebesgue almost everywhere on each of the intervals $[jq^{-m},(j+1)q^{-m})$, $j=0,1,...,q^{m}-1$,  
then for $n=m,m+1,...$, $P_n=\mu$  and $(X_1,...,X_n)$ is independent of $T^n(X)$.
\ep

\begin{proof}
If $P_m=\mu$ then by invariance of $\mu$ under $T$, $P_n=\mu$ for $n=m,m+1,...$. 
Let $K=\sum_{i=1}^mX_iq^{m-i}$ and $U=T^m(X)$, so $X=(K+U)q^{-m}$.
For Lebesgue almost every $t\in[0,1)$, 
\[f(t)=q^m\mathrm P(K=\lfloor q^m t\rfloor\,|\,U=q^mt-\lfloor q^mt\rfloor)f_m(q^mt-\lfloor q^mt\rfloor)\]
since
\begin{align*}
&F(t)=\mathrm P((K+U)q^{-m}\le t)\\
&= F(q^{-m}\lfloor q^mt \rfloor) + \int_0^{q^m t-\lfloor q^mt\rfloor}\mathrm P(K = \lfloor q^mt\rfloor\,|\,U=u)f_m(u)\,\mathrm du.
\end{align*}
Thereby the first assertion follows.  

Suppose that $c_j$ is a constant and $f=c_j$  Lebesgue almost everywhere on  $[jq^{-m},(j+1)q^{-m})$ for $j=0,1,...,q^{m}-1$. Then 
$$\sum_{j=0}^{q^m-1}c_jq^{-m}=\sum_{j=0}^{q^m-1}\int_{jq^{-m}}^{(j+1)q^{-m}}c_j=\int_0^1 f=1,$$ 
and so for Lebesgue almost all $x\in[0,1)$, \eqref{e:5555} gives that $f_m(x)=1$. Therefore, $P_m=\mu$, and hence $P_n=\mu$ for $n=m,m+1,...$. Consequently, the last assertion follows from
\eqref{e:fkuqm}, using that $\sum_{i=1}^m X_iq^{m-i}$ and $(X_1,...,X_m)$ are in a one-to-one correspondence.
\end{proof}

\section{Couplings}\label{s:couplings}


Let $f$ be a PDF on $[0,1)$. We introduce the following notation. Let $I_\emptyset=I_{1;0}=[0,1)$ and $c_\emptyset=c_{1;0}=\inf_{I_\emptyset}f$. For $n=1,2,...$ and $x_1,x_2,...\in\{0,1,...,q-1\}$,  let $k=1+\sum_{i=1}^n x_iq^{n-i}$ and
	$$I_{x_1,...,x_n}=I_{k;n}=[(k-1)q^{-n},k q^{-n})$$ 
	and
	$$c_{x_1,...,x_n}={c_{k;n}=}\inf_{I_{x_1,...,x_n}} f-\inf_{I_{x_1,...,x_{n-1}}} f.$$ 

Recall that a function $f$ is lower semi-continuous at a point $x$ if for any sequence $y_n\to x$, it holds that  $\liminf_n {f(y_n)} \geq f(x)$.
Note that if $x=\sum_{n=1}^\infty x_nq^{-n}\in[0,1)$ is not a base-$q$ fraction, then lower semi-continuity at $x$ is equivalent to 
\begin{equation}\label{e:condition}
f(x)=\lim_{n\rightarrow\infty}\inf_{y\in I_{x_1,...,x_n}} f(y).
\end{equation}
Write $U\sim\mu$ if $U$ is a uniformly distributed random variable on $[0,1)$. 

\begin{theorem}\label{t:coupling} Suppose $f$ is  lower semi-continuous at Lebesgue almost all points in $[0,1)$. Then
there is a coupling between  
$X \sim f$ and a  non-negative integer-valued random variable $N$ 
such that 
$T^N(X)\sim\mu$ is independent of $(X_1,...,X_N)$.
\end{theorem}

\begin{remark} Set $\{0,1,...,q-1\}^0=\{\emptyset\}$ so we interpret $\emptyset$ as no digits. Then $(X_1,...,X_N)$ is a discrete random variable with state space $\cup_{n=0}^\infty\{0,1,...,q-1\}^n$.

	Commonly used PDFs are lower semi-continuous almost everywhere. For an example where this condition does not hold, let
$0<\epsilon_{k;n}<q^{-n}$ such that \am{$a=\sum_{n=0}^\infty\sum_{k=1}^{q^n}\epsilon_{k;n}<0$}. Further, let $J_{k;n}=[(k-1)q^{-n},(k-1)q^{-n}+\epsilon_{k;n}]$, $G=\cup_{n=0}^\infty\cup_{k=1}^{q^n}J_{k;n}$, and $H=[0,1)\setminus G$. Then $H$ is a Borel set with $0<\mu(H)\le1$, since $\mu(J_{k;n})=\epsilon_{k;n}$ and so $\mu(G)\le a<1$. Hence, the uniform distribution $\mu_H$ on $H$ is absolutely continuous. 
Since $H$ contains no base-$q$ fraction and the set of base-$q$ fractions is dense in $[0,1)$, any interval will contain points not in $H$. 
Now, any  PDF $f$ for $\mu_H$ will be zero outside $H\cup A$ for some nullset $A$ (depending on the version of $f$), so for all integers $n\ge0$ and $1\le k\le q^n$, $f$ will be zero on ${I_{k;n}}\setminus (H\cup A)\ne\emptyset$. Thus the right hand side in \eqref{e:condition} is zero, so $f$ is not lower semi-continuous anywhere.
\end{remark}

\begin{proof} 
For Lebesgue almost all $x=\sum_{n=1}^\infty x_nq^{-n}\in[0,1)$ with $x_n=\lfloor T^{n-1}(x)q\rfloor$, assuming $x$ is not a base-$q$ fraction (recalling that the set of base-$q$ fractions is a Lebesgue nullset), \eqref{e:condition} gives
	 \begin{align}
	 f(x)&=\inf_{I_\emptyset}f+\left(\inf_{I_{x_1}}f-\inf_{I_\emptyset}f\right)+\left(\inf_{I_{x_1,x_2}}f-\inf_{I_{x_1}}f\right)+...\nonumber\\
	 &=\sum_{n=0}^\infty c_{x_1,...,x_n} . 
	 \label{e:hurra}
	 \end{align}
   Let $N$ be a random variable such that for $f(x)>0$, conditionally on $X=x$, 
 \[\mathrm P(N=n\,|\,X=x)=c_{x_1,...,x_n}/f(x),\quad n=0,1,...\]
 By \eqref{e:hurra} and since $c_{x_1,...,x_n}\ge0$, this is a well-defined conditional distribution.

 By Bayes theorem, conditioned on $N=n$ with $\mathrm P(N=n)>0$, $X$  follows an absolutely continuous distribution with PDF 
 \[f(x\,|\,n)=c_{x_1,...,x_n}/\mathrm P(N=n).\]
Therefore, since $f(x|n)$ is constant on each of the intervals $I_{k;n}$, conditioned on $N=n$ we immediately see that $(X_1,...,X_n)$ (interpreted as nothing if $n=0$) and $T^n(X)$ are independent and that $T^n(X)\sim\mu$. The latter implies that $T^N(X)\sim \mu$ is independent of $N$. Consequently, if we do not condition on $N$, we have that $(X_1,...,X_N)$ and $T^N(X)\sim\mu$ are independent.
\end{proof}

\begin{corollary}\label{c:latest}
For the coupling construction in the proof of Theorem~\ref{t:coupling}, conditioned on $X=x$ with $f(x)>0$, we have
\begin{equation}\label{e:latest}
\mathrm P(N\le n\,|\,X=x)=\sum_{k=0}^n c_{x_1,...,x_k}/f(x),\quad n=0,1,...,
\end{equation}
where $x_k=\lfloor T^{k-1}(x)q\rfloor$ for $1\le k\le n$.
Moreover,
\begin{equation}\label{e:PNn}
\mathrm P(N\le n)=q^{-n}\sum_{k=1}^{q^n} \inf_{I_{k;n}}f,\quad n=0,1,...
\end{equation}
\end{corollary}

\begin{remark} Corollary~\ref{c:latest} is used  in Section~\ref{s:examples} to quantify how many digits are needed to make the remainder uniformly distributed with sufficiently high probability. 

Since a PDF is only defined up to a set of measure zero, it is possible for a distribution to have several PDFs that are almost everywhere lower semi-continuous but give rise to different constants $c_{x_1,\ldots,x_n}$. Hence  the distribution of $(X_1,\ldots,X_N)$ is not uniquely defined. For example, if $X\sim\mu$, letting $f$ be the indicator function on $[0,1)$ gives $N=0$ almost surely, whilst 
	letting $f$ be the indicator function on $[0,1)\setminus\{x_0\}$ for some $x_0\in[0,1)$ gives $\mathrm P(N\le n)=1-q^{-n}$. By \eqref{e:PNn}, in order to make $N$ as small as possible, we prefer a version of $f$ which is as large as possible.
\end{remark}

\begin{proof}
The proof of Theorem~\ref{t:coupling} gives immediately \eqref{e:latest}. Thus, for $n=0,1,...$,
\[
\mathrm P(N\le n)=\int_0^1 \mathrm P(N\le n\,|\,X=x)f(x)\,\mathrm dx
=\sum_{k=0}^n \sum_{j=1}^{q^k} c_{j;k}q^{-k}.
\]
So $\mathrm P(N=0)=c_\emptyset$ in agreement with \eqref{e:PNn}. For $n=1,2,...$, we have
\begin{align*}
\sum_{k=0}^n \sum_{j=1}^{q^k} c_{j;k}q^{-k}
&=c_\emptyset+\sum_{k=1}^n \sum_{(x_1,...,x_k)\in\{0,1,...,q-1\}^k}c_{x_1,...,x_k}q^{-k}\\
 &=\inf_{I_\emptyset}f+\sum_{x_1\in\{0,1,...,q-1\}}\left(\inf_{I_{x_1}}f-\inf_{I_\emptyset}f\right)q^{-1}+...\\
	&+\sum_{(x_1,...,x_n)\in\{0,1,...,q-1\}^n}\left(\inf_{I_{x_1,...,x_n}}f-\inf_{I_{x_1,...,x_{n-1}}}\right)q^{-n}\\
	&=q^{-n}\sum_{(x_1,...,x_n)\in\{0,1,...,q-1\}^n}\inf_{I_{x_1,...,x_n}}f\\
	&=q^{-n}\sum_{j=1}^{q^n} \inf_{I_{j;n}}f. 
\end{align*}
Thereby \eqref{e:PNn} follows.
\end{proof}

\begin{corollary}\label{c:new}
Let the situation be as in Theorem~\ref{t:coupling}. The output of  
the following simulation algorithm  is distributed as $X\sim f$:
\begin{enumerate}
\item[(a)] Draw
$N$ from \eqref{e:PNn}.
\item[(b)] Conditionally on $N$, generate
a discrete random variable $K$ with 
 \begin{equation}\label{e:def-cond}
	\mathrm P(K=\am{k-1}\,|\,N=n)\propto c_{k;n},\quad k=1,...,q^n,\ n=0,1,...
	\end{equation}
	\item[(c)] Independently of $(N,K)$ pick a random variable $U\sim\mu$.
\item[(d)]  Output $(K+U)q^{-N}$.
\end{enumerate}
\end{corollary}

\begin{proof}
Let $a_n=\sum_{k=1}^{q^n}c_{k;n}$ be the normalizing constant in \eqref{e:def-cond}.
Conditioned on $N=n$ with $\mathrm P(N=n)>0$, steps (b) and (c) give that $U\sim\mu$ and $K$ are independent, so the conditional distribution of
	  $(K+U)q^{-N}$ is absolutely continuous with a conditional PDF given by
\[f(x\,|\,n)=q^{n}c_{k;n}/a_n\quad\mbox{if $x\in I_{k;n}$}.\]
	 \am{Moreover, we get from \eqref{e:PNn} that
$\mathrm P(N=0)=c_\emptyset$ and
\[\mathrm P(N=n)=\mathrm P(N\le n)-\mathrm P(N<n)=a_n q^{-n},\quad n=1,2,...\]}
Therefore, the (unconditional) distribution of $(K+U)q^{-N}$ is absolutely continuous with a PDF which at each point $x=\sum_{n=1}^\infty x_nq^{n}\in[0,1)$ with $x_n=\lfloor T^{n-1}(x)q\rfloor$ is given by
\[\sum_{n=0}^\infty f(x\,|\,n)\mathrm P(N=n)=\sum_{n=0}^\infty q^{n}\left(c_{x_1,...,x_n}/a_n\right)a_n q^{-n}=\sum_{n=0}^\infty c_{x_1,...,x_n}.\]
This PDF agrees with \eqref{e:hurra}, so $(K+U)q^{-N}\sim f$.
\end{proof}

	Denote by $\mathcal{B}$ the class of Borel subsets of $[0,1)$.
	The total variation distance between two probability measures $\nu_1$ and $\nu_2$ defined on $\mathcal{B}$ and
	with PDFs $g_1$ and $g_2$, respectively,  is given by
	\begin{equation}\label{e:TV2}
	d_{\mathrm{TV}}(\nu_1,\nu_2)= \sup_{A\in \mathcal{B}} |\nu_1(A) - \nu_2(A)| =
	\frac12\|g_1-g_2\|_1,
	\end{equation}
	see e.g.\ Lemma~2.1 in \cite{Tysbakov}. Then Theorem~\ref{t:coupling} shows the following.

\begin{corollary}\label{c:1} Let the situation be as in Theorem~\ref{t:coupling}. Then
\begin{equation}\label{e:ci}
 d_{\mathrm{TV}}(P_n,\mu)\le \mathrm P(N> n),\quad n=0,1,...
 \end{equation}
\end{corollary}

\begin{remark}\label{r:1} 
In general the coupling inequality \eqref{e:ci} is sharp:  
For $n=0,1,...$, let $b_n=1-d_{\mathrm{TV}}(P_n,\mu)=\int_0^1\min\{1,f_n(t)\}\,\mathrm dt$ (with $f_0=f$). 
 It is well-known that $b_n$ is the maximal number such that there exists a coupling between $T^n(X)\sim P_n$ and a uniform random variable $U\sim \mu$ for which $T^n(X)=U$ with probability $b_n$ 
 (see e.g.\ Theorem 8.2 in \cite{thorisson}). Thus $d_{\mathrm{TV}}(P_n,\mu)= \mathrm P(N> n)$ if and only if $\int_0^1\min\{1,f_n(t)\}\,\mathrm dt=q^{-n}\sum_{k=1}^{q^n} \inf_{I_{k;n}}f$. In particular, $d_{\mathrm{TV}}(P_0,\mu)= \mathrm P(N> 0)$ if and only if $X\sim\mu$.
 
 It follows from Corollary~\ref{c:latest} and \ref{c:1} that \eqref{e:condition} implies $\lim_{n\rightarrow\infty} d_{\mathrm{TV}}(P_n,\mu)=0$. In Theorem~\ref{t:1} below we show that \eqref{e:condition} is not needed for this convergence result.
\end{remark}

\begin{proof} Using Corollary~\ref{c:new}, let
$X=(K+U)q^{-N}$. For $n=0,1,...$, if $Q_n$ denotes the probability distribution of $T^n(U)$, then $Q_n=\mu$, and so
\[d_{\mathrm{TV}}(P_n,\mu) = d_{\mathrm{TV}}(P_n,Q_n) \le \mathrm P(T^n(X) \ne T^n(U))\le\mathrm P(N>n),\]  
where the first inequality is the standard coupling inequality for the coupled random variables $T^n(X)$ and $T^n(U)$, and the last inequality follows since $N\le n$ implies $T^n(X)=T^n(U)$.
  Thereby \eqref{e:ci} is verified.
\end{proof}

\begin{remark}\label{r:Wass} 
By the Kantorovich-Rubinstein theorem, the Wasserstein distance between two probability measures $\nu_1$ and $\nu_2$ on $[0,1]$ is given by 
\begin{equation*}
	W_1(\nu_1,\nu_2)= \inf_{\gamma \in \Gamma (\nu_1,\nu_2)} \{ \mathrm E |Y_1 - Y_2| \mid (Y_1,Y_2)\sim \gamma\},
\end{equation*}
where $\Gamma(\nu_1,\nu_2)$ consists of all couplings of $\nu_1$ and $\nu_2$. By \cite[Thm 4]{GibbsSu}, 
\begin{equation*}
W_1(\nu_1,\nu_2) \leq d_{\mathrm{TV}}(\nu_1,\nu_2), 
\end{equation*}
so by Remark \ref{r:1}, Corollary \ref{c:1} implies 
$$W_1(P_n,\mu)\le \mathrm P(N>n)\to 0.$$ 
The latter bound can be improved by using the coupling between $T^n(X)$ and $T^N(X)\sim \mu$ to obtain
\begin{align*}
	W_1(P_n,\mu){}& \leq \mathrm E | (T^n(X)- T^N(X))1_{N>n}|\\
	 &\leq \int_0^1 \max\{|x|,|x-1|\} \mathrm{d} x \ \mathrm P(N>n) \\
	 &= \frac{1}{4}\mathrm P(N>n).
\end{align*}
 See also \cite{GibbsSu} for an overview of the relation between the total variation distance and other measures of distance between probability measures.
\end{remark}

\section{Asymptotic results} 
\label{s:results-q}

We need some notation for the following theorem. For a real, measurable function $g$ defined on $(0,1)$, denote its $L_1$- and supremum-norm
by
$\|g\|_1=\int_0^1|g(t)|\,\mathrm dt$ and
$\|g\|_\infty=\sup_{x\in(0,1)}|g(x)|$, respectively, and denote the corresponding $L_1$-space by $L_1(0,1)=\{g\,|\,\|g\|_1<\infty\}$
(here, $\|g\|_\infty$  may be infinite when there are no further assumptions on $g$). Let $\bar L_1(0,1)=\{g\,| \int_0^1g(t) dt =1,\allowbreak \|g\|_1 < \infty\}$ be the subset of functions with finite $L_1$-norm and integral over $[0,1]$ equal one, and $\bar L_1'(0,1)\subset \bar L_1(0,1)$ its subset of  differentiable functions $g$ such that $\|g'\|_\infty<\infty$. For $g\in \bar L_1'(0,1)$, $n\in\mathbb N$, $j=0,1,...,q^n-1$, and $0<x<1$, define $g_{n,j}'(x)=g'(x)$ if $q^{-n}j<x<q^{-n}(j+1)$ and $g_{n,j}'(x)=0$ otherwise, and define
\begin{equation}\label{gdef}
g_n(x) = q^{-n}\sum_{j=0}^{q^n-1} g(q^{-n}(j+x)).
\end{equation}
Henceforth, we also think of $f$ as an element of $\bar L_1(0,1)$. 

\bt\label{t:1} If  $f \in \bar L_1(0,1)$ and $g \in \bar L_1'(0,1)$ then
\begin{equation}\label{e:uniform1}
d_{\mathrm{TV}}(P_n,\mu) 
\le \frac{1}{2}\|f-g\|_1+ \frac{1}{6}q^{-2n} \sum_{j=0}^{q^n-1}\|g_{n,j}'\|_\infty
\le \frac{1}{2}\|f-g\|_1+\frac{1}{6}q^{-n} \|g'\|_\infty.
\end{equation}
In particular,
\begin{equation}\label{e:uniform2}
\lim_{n\rightarrow\infty} d_{\mathrm{TV}}(P_n,\mu) 
=0
\end{equation}
and we have the following sharper convergence results.
If $f \in \bar L_1'(0,1)$ then $P_n$ converges exponentially fast:
\begin{equation}\label{e:uniform3}
d_{\mathrm{TV}}(P_n,\mu) \le \frac{1}{6}q^{-2n} \sum_{j=0}^{q^n-1}\|f_{n,j}'\|_\infty
\le \frac{1}{6}q^{-n} \|f'\|_\infty.
\end{equation}
If $\|f\|_\infty < \infty $ and $f$ is continuous except for finitely many points, then 
\begin{equation} \label{fcont}
|f_n(x)-1| \to 0 \quad \mbox{uniformly for $x\in(0,1)$}.
\end{equation}
If $f$ is twice differentiable with $\|f''\|_\infty<\infty$ then we have the following improvement of \eqref{e:uniform3}:
\begin{align} \label{twicediff}
d_{\mathrm{TV}}(P_n,\mu)&= \frac{1}{8} q^{-2n} \left|\sum_{j=0}^{q^n -1} f'(\xi_{nj})\right| + O(\|f''\|_\infty q^{-2n}) \nonumber\\
&\le \frac{1}{8} q^{-n} \|f'\|_\infty + O(\|f''\|_\infty q^{-2n})
\end{align}
where $\xi_{n,j} \in (q^{-n}j,q^{-n}(j+1))$ is arbitrary.
\et

Before proving this theorem we need the following lemma.

\bl\label{technical}
	Let $f \in \bar L_1(0,1)$, $g \in \bar L_1'(0,1)$. For every $x\in(0,1)$,
	\begin{equation}\label{e:g_n-1 bound}
	|g_n(x)-1| \le q^{-2n} \left(x^2-x + \frac{1}{2} \right)\sum_{j=0} ^{q^n-1} \|g'_{n,j}\|_\infty \le \frac{1}{2}q^{-n}\|g'\|_\infty,
	\end{equation}
	and 
	\begin{equation}\label{e:0000}
	\int_0^1 |f_n(x) -g_n(x)|\,\mathrm dx \le \|f-g\|_1.
	\end{equation}
	If $g$ is twice differentiable on $(0,1)$ with $\|g''\|_\infty<\infty$ then for every $x\in(0,1)$,
	\begin{equation}\label{gtwicediff}
	g_n(x) - 1 = q^{-2n}\left(x-\frac{1}{2}\right) \sum_{j=0} ^{q^n-1}  g'(\xi_{n,j}) + O(\|g''\|_\infty q^{-2n}),
	\end{equation}
	where each $\xi_{n,j} \in (q^{-n}j,q^{-n}(j+1))$ is arbitrary.
\el

\begin{remark}
	Of course, (\ref{e:g_n-1 bound})   and (\ref{gtwicediff}) hold with $g_n$ replaced by $f_n$ if $f$ is differentiable respectively twice differentiable with $\|f''\|_\infty < \infty.$  For Example~\ref{ex:2} below it is useful to realize that in (\ref{gtwicediff}), $q^{-n} \sum_{j=0}^{q^n-1}g'(\xi_{n,j})$ is a Riemann sum for the integral $\int_0^1 g'(t) \mathrm dt$.
\end{remark}

\begin{proof}
	Let $x\in(0,1)$.    From \eqref{gdef} we have
	\begin{align} \nonumber
	g_n(x) - 1 &= \sum_{j=0} ^{q^n-1} \int_{q^{-n}j}^{q^{-n}(j+1)} [g(q^{-n}(j+x)) - g(t)]\, \mathrm d t \\
	& = \sum_{j=0} ^{q^n-1} \int_{0}^{1} q^{-n} [g(q^{-n}(j+x)) - g(q^{-n}(j+y))]\, \mathrm d y.\label{dens_diff}
	\end{align}
	If $g$ is differentiable on $(0,1)$ with $\|g'\|_\infty<\infty$, we get
	by the mean value theorem, 
	$$ |g(q^{-n}(j+x)) - g(q^{-n}(j+y))| \le \|g'_{n,j}\|_\infty q^{-n}|x-y|,$$
	which yields the bound
	\begin{equation} \label{e:gn_diff}
	|g_n(x) - 1| \leq q^{-2n}\int_{0}^{1} |x-y|\, \mathrm  d y \sum_{j=0} ^{q^n-1} \|g'_{n,j}\|_\infty = q^{-2n}\left(x^2-x+\frac{1}{2}\right) \sum_{j=0} ^{q^n-1} \|g'_{n,j}\|_\infty. 
	\end{equation}
	Thereby \eqref{e:g_n-1 bound} follows.
	Moreover,
		\begin{align} \nonumber
		\int_0^1|f_n(x) - g_n(x)|\,\mathrm dx 
		&\le \sum_{j=0}^{q^n-1} \int_0^1q^{-n}|f(q^{-n}(j+x)) -g(q^{-n}(j+x))|\,\mathrm dx\\ &= \|f-g\|_1
		\label{fn-gn}
		\end{align} 
	whereby \eqref{e:0000} follows.
	If $g$ is twice differentiable on $(0,1)$ with $\|g''\|_\infty<\infty$, the mean value theorem gives
	$$ g(q^{-n}(j+x)) - g(q^{-n}(j+y)) = g'(\xi_{x,y}) q^{-n} (x-y) =  g'(\xi_{n,j}) q^{-n} (x-y) + O(\|g''\|_\infty q^{-2n}), $$
	where $\xi_{x,y} \in (q^{-n}j,q^{-n}(j+1))$ depends on $x$ and $y$ and $\xi_{n,j} \in (q^{-n}j,q^{-n}(j+1))$ is arbitrary.  The second equality was obtained by applying the mean value  theorem to $g'(\xi_{x,y})-g'(\xi_{n,j})$. 
	Inserting this into \eqref{dens_diff} yields
	\[g_n(x) - 1 = q^{-2n}\sum_{j=0} ^{q^n-1}  g'(\xi_{n,j})  \int_{0}^{1}  (x-y)\, \mathrm d y + O(\|g''\|_\infty q^{-2n}), \]
	which reduces to \eqref{gtwicediff}.
\end{proof}  

We are now ready for the proof of Theorem \ref{t:1}.

\begin{proof} 
	We have
		\begin{align} \nonumber
		d_{\mathrm{TV}}(P_n,\mu) &= \frac{1}{2} \int_0^1 |f_n(x) - 1|\,\mathrm dx\\ 
		& \le \frac{1}{2}\int_0^1 |f_n(x) - g_n(x)|\,\mathrm dx+ \frac{1}{2}\int_0^1 |g_n(x) -1|\,\mathrm dx \nonumber\\ 
		&\le \frac{1}{2}\|f-g\|_1 + \frac{1}{6}q^{-2n}\sum_{j=0}^{q^n-1}\|g'_{nj}\|_\infty \label{e:dTVfg} \\
		& \le \frac{1}{2}\|f-g\|_1 + \frac{1}{6}q^{-n}\|g'\|_\infty \nonumber
		\end{align}
		where we get the equality from \eqref{e:TV2} and 
		the second inequality from \eqref{e:g_n-1 bound}, \eqref{e:0000}, and since $\int_0^1\left(x^2-x +1/2\right)\,\mathrm dx = 1/3$.
		Thereby \eqref{e:uniform1} is verified.	
	Taking $n\to \infty $ in \eqref{e:uniform1} and using that $\bar L'_1(0,1)$ is dense in $\bar L_1(0,1)$, we get \eqref{e:uniform2}.
	Equation (\ref{e:uniform3}) follows from \eqref{e:uniform1}  by setting $g = f$.
	
		For  the proof of \eqref{fcont} we suppose $f$ is continuous except at  $x_1,\ldots,x_m \in (0,1)$  and set $x_0=0$ and $x_{m+1}=1$. Let $\delta>0$ and 
		\begin{align*}
		I_n&=\{j \in \{ 0,1,\ldots,q^n-1\} \mid \exists i \in \{0,1,\ldots,m+1\}: | q^{-n}j - x_i | < \delta \},\\
		J_n&= \{ 0,1,\ldots,q^n-1\}\backslash I_n.
		\end{align*}
		By \eqref{dens_diff}, 
		\begin{align}
		\nonumber
		f_n(x)-1 &= \sum _{j\in I_n} \int_{ q^{-n}j}^{q^{-n}(j+1)} (f(q^{-n}(j+x)) -f(t))\,\mathrm dt\\
		&+ \sum _{j\in J_n} \int_{ q^{-n}j}^{q^{-n}(j+1)} (f(q^{-n}(j+x)) -f(t))\,\mathrm dt. \label{e:disc_proof}
		\end{align}
		Given $\varepsilon >0$, we choose $\delta$ so that  $\delta < \varepsilon / (6(m+2)\|f\|_\infty)$. Then, since the cardinality of $I_n$ is at most $(m+2)(2q^n \delta + 1)$, the first sum in \eqref{e:disc_proof} is bounded by
		\[\left(\frac{2q^n\varepsilon}{6\|f\|_\infty}+m+2\right)q^{-n}\|f\|_\infty=\frac{\varepsilon}{3}+(m+2) q^{-n}\|f\|_\infty<\frac{\varepsilon}{2}\]
		for $n$ sufficiently large. Moreover, for $n$ large enough, the second sum in \eqref{e:disc_proof} is bounded by $\varepsilon/2$ since $f$ is uniformly continuous on $(0,1)\backslash \bigcup_{i=0}^{m+1} (x_i -\delta/2, x_i+\delta/2)$, which is a closed set. Thus, for large enough $n$, $|f_n(x) - 1|< \varepsilon$ which gives \eqref{fcont} since $\varepsilon >0$ is arbitrary.
	
	To prove (\ref{twicediff}) we use (\ref{gtwicediff}) with $g$ replaced by $f$.
	Then, for every $A\in\mathcal B$,
	$$\int_A (f_n(t) -1)\,\mathrm dt = q^{-2n}\int_A\left(t-\frac{1}{2}\right)\,\mathrm dt \sum_{j=0}^{q^n-1}f'(\xi_{nj}) + O(\|f''\|_\infty q^{-2n}).$$
	 We have
		$$\sup_{A\in \mathcal{B} }\bigg|\int_A\left(t-\frac{1}{2}\right)\,\mathrm dt\bigg| = \frac{1}{2}\sup_{A\in \mathcal{B} }\bigg|\int_A|2t-1|\,\mathrm dt\bigg| = \frac{1}{4} \int_0^1 (2t-1)\,\mathrm dt = \frac{1}{8}$$
		where the second identity follows from \eqref{e:TV2}. This gives
		 (\ref{twicediff}).
\end{proof}

\begin{remark} In continuation of Remark~\ref{r:1}, by
 Theorem~\ref{t:1}, $b_n\rightarrow 1$ and under weak conditions the convergence is exponentially fast. 
\end{remark}

\section{So how many digits are needed?}\label{s:examples}

This section starts with some theoretical statistical considerations and continues then with some specific examples.

Consider a parametric model for the probability distribution of $X$ 
given by a parametric class of lower semi-continuous densities $f_\theta$ where $\theta$ is an unknown parameter. By Theorem~\ref{t:coupling} this specifies a parametric model for $(X_1,...,X_N)$ which is independent of $T^N(X)\sim\mu$. In practice we cannot expect $N$ to be observable, but let us imaging it is. Then, according to general statistical principles (see e.g.\ \cite{B-N}), statistical inference for $\theta$  should be based on the  sufficient statistic
$(X_1,...,X_N)$, whilst $T^N(X)$ is an ancillary statistic and hence contains no information about $\theta$. Moreover, Theorem~\ref{t:1} ensures (without assuming that the densities are lower semi-continuous) that  $T^n(X)$ is approximately uniformly distributed. Hence, if $n$ is `large enough', nearly
all information about $\theta$ is contained in $(X_1,...,X_n)$.  

\begin{remark}
For another paper it could be interesting to consider a so-called missing data approach for a parametric model of the distribution of $(X_1,...,X_N)$, with an unknown parameter $\theta$ and 
treating $N$ as an unobserved statistic (the missing data): Suppose $X^{(1)},...,X^{(k)}$ are IID copies of $X$, with corresponding `sufficient statistics' 
$(X^{(i)}_1,...,X^{(i)}_{N^{(i)}})$, $i=1,...,k$. The EM-algorithm may be used for estimation of $\theta$. Or a Bayesian approach may be used, imposing a prior distribution for $\theta$ and then considering the posterior distribution of $(N^{(1)},...,N^{(k)},\theta)$.
\end{remark}

\jm{According to Corollary~\ref{c:latest}, the number of digits we need will in general depend on the realization of $X=x$. As a measure for this dependence, for $f(x)>0$ and $n=0,1,...$, we may consider $\mathrm P(N>n\,|\,X=x)$ as a function of $x$, which can be calculated from \eqref{e:latest}. Since $N\le n$ implies $T^n(X)\sim\mu$,
an overall measure which quantifies the number $n$ of digits needed is given by $\mathrm P(N>n)$, cf.\  \eqref{e:PNn}. The use of these measures requires that $f$ is lower semi-continuous, whilst the bounds in Theorem~\ref{t:1} for the total variation distance $d_{\mathrm{TV}}(P_n,\mu)$ hold without this condition.
The following Examples~\ref{ex:1} and \ref{ex:2} demonstrate how these measures 
can be used to quantify the number $n$ of digits needed in order that $N> n$ (conditioned or not on $X=x$) 
with a small probability or that $d_{\mathrm{TV}}(P_n,\mu)$ is small. }

\begin{example}\label{ex:1}
	Any number $y\not=0$ can uniquely be written as $y=sq^k(y_0+y_f)$ where $s=s(y)\in\{\pm1\}$ is the sign of $y$, $k=k(y)\in\mathbb Z$ determines the decimal point of $y$ in base-$q$, $y_0=y_0(y)\in\{1,...,q-1\}$ is the leading digit of $y$ in base-$q$, and $y_0+y_f$ is the so-called significand of $y$ in base-$q$, where $y_f=y_f(y)\in[0,1)$ is the fractional part of $y_0+y_f$ in base-$q$. 
	Correspondingly, 
	consider any real-valued random variable $Y\not=0$ (or just $\mathrm P(Y=0)=0$), so (almost surely) $Y=Sq^K(X_0+X)$ where
	$S=s(Y)$, $K=k(Y)$, $X_0=y_0(Y)$, and $X=y_f(Y)$ are random variables. Let $X_1,X_2,...$ be the digits of $X$ in the base-$q$ expansion, cf.\ \eqref{e:q}. We call $X_0,X_1,X_2,...$ the significant digits of $Y$ in base-$q$. 
	By definition $Y$ satisfies the extended Newcomb-Benford law if
	\begin{equation}\label{e:Benford}
	\mathrm P(X_0=x_0,...,X_n=x_n)=\log_q\left(1+1\bigg/\sum_{j=0}^n q^{n-j}x_j\right)
	\end{equation}
	for $n=0,1,...$ and any $x_0\in\{1,...,q-1\}$ and $x_j\in\{0,1,...,q-1\}$ with $1\le j\le n$. Equivalently, the log-significand of $Y$ in base-$q$, $\log_q(X_0+X)$, is uniformly distributed on $[0,1)$ (Theorem 4.2 in \cite{Berger}). 
Then $X$ has CDF and PDF given by
	\begin{equation}\label{e:f-Benford}
	F(x)=\sum_{j=1}^{q-1}\left(\log_q(j+x)-\log_qj\right),\quad f(x)=\sum_{j=1}^{q-1} \frac{1}{\ln q}\frac{1}{j+x},
	\end{equation}
	for $0\le x\le 1$. 
	
	The extended Newcomb-Benford law applies to a wide variety of real datasets, see \cite{Hill, Berger} and the references therein. 
	The law is equivalent to appealing scale invariance properties: Equation 
	\eqref{e:Benford} is equivalent to that $Y$ has scale invariant significant digits (Theorem~5.3 in \cite{Berger}) 
	or just that
	there exists some $d\in\{1,...,\allowbreak q-1\}$ such that $\mathrm P(y_0(aY)=d)$ does not depend on $a>0$ 
	(Theorem~5.8 in \cite{Berger}). 
	Remarkably, for any positive random variable $Z$ which is independent of $Y$, if the extended Newcomb-Benford law is satisfied by $Y$, it is also satisfied by $YZ$ (Theorem~8.12 in \cite{Berger}). 
	
	For the remainder of this example, suppose \eqref{e:Benford} is satisfied.  Considering  \eqref{e:PNn} gives for $n=0,1,...$ that
\[\mathrm P(N\le n)=\frac{q^{-n}}{\ln q}\sum_{j=1}^{q-1}\sum_{k=1}^{q^n}\frac{1}{j+kq^{-n}}.\]
\jm{The tail probabilities $\mathrm P(N>n)$ decrease quickly as $n$ and $q$ increase, see the left panel in Figure~\ref{fig:P(N>n)} for plots of $\mathrm P(N>n)$ against $n$ for $q=2,3,5,10$.} {\color{black}The middle panel of Figure~\ref{fig:P(N>n)} shows $\mathrm P(N>1\,|\,X=x)$ as a function of $x$ for $q=10$. We see large fluctuations, with probabilities dropping to zero when approaching the right limit of the intervals $I_{k;1}$, where $\inf_{I_{k;1}} f $ is attained. To avoid these fluctuations, the right panel of Figure~\ref{fig:P(N>n)} shows an upper bound on $\mathrm P(N>n\,|\,X=x)$ as a function of $x$ for $q=10$ and $n=0,1,2,3$. The upper bound is found by noting that on each $I_{k;n}$, $\mathrm P(N>n\,|\,X=x)$ is convex decreasing towards zero. Hence an upper bound is given by evaluating at the left end points and interpolating linearly. The plot shows that $\mathrm P(N>n\,|\,X=x)$ is very close to zero for all $x$ already for $n=2$.}
\begin{figure}\label{fig:P(N>n)}
	\begin{center}
	\includegraphics[width=\textwidth]{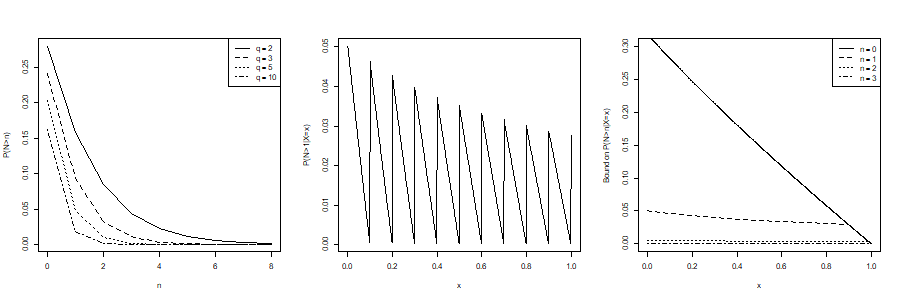}
	\end{center}
\caption{{\color{black}Left panel: $\mathrm P(N>n)$ as a function of $n$ for $q=2,3,5,10$. Middle panel: $\mathrm P(N>1\,|\,X=x)$ as a function of $x$ for $q=10$.
Right panel: An upper bound for $\mathrm P(N>n\,|\,X=x)$ as a function of $x$ for $n=0,1,2,3$ and $q=10$.}}
\end{figure}	
	
This  is also in accordance with Theorem~\ref{t:1} stating that $T^n(X)$ converges to a uniform distribution on $[0,1)$ and hence the first digit $X_n$ of $T^n(X)$ is approximately uniformly distributed on $\{0,1,...,q-1\}$ when $n$ is large.	For $n=1,2,...$ and $x_n\in\{0,1,...,q-1\}$, we have 
	\begin{equation*}
	\mathrm P(X_n=x_n)
	=\log_q\left(\prod_{j=1}^{q-1}\prod_{i=1}^{q^{n-1}}\left(1+\frac{1}{jq^n+(i-1)q + x_n}\right)\right)
	\end{equation*}
	where $\mathrm P(X_n=x_n)$ is a decreasing function of $x_n$. 
	The left part of Figure~\ref{fig:benford} shows plots of 
	$\mathrm P(X_n=0)-\mathrm P(X_n=q-1)$ versus $n$ for  $q=2,3,5,10$ indicating fast convergence to uniformity and that the convergence speed increases with $q$. The right part of Figure~\ref{fig:benford} illustrates the stronger statement in \eqref{fcont} that the PDF $f_n$ of $T^n(X)$ converges uniformly to the uniform PDF.
	\begin{figure}
		\includegraphics[width=\textwidth]{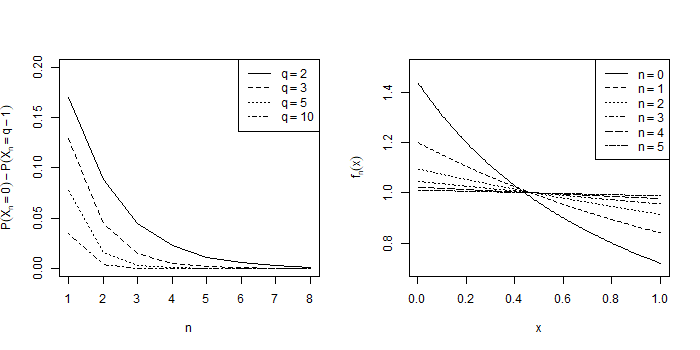}
		\caption{Left panel: $\mathrm P(X_n=0)-\mathrm P(X_n=q-1)$ as a function of $n$ for various values of $q$  when $f$ is as in \eqref{e:f-Benford}. Right panel: $f_n$ when $q=2$ and $n=0,\ldots,5$.}\label{fig:benford}
	\end{figure}
	
	To further illustrate the fast convergence, we drew a sample of 1000 observations with CDF \eqref{e:f-Benford} and made a $\chi^2$ goodness-of-fit test for uniformity of $X_n$. Considering a significance level of 0.05, the rejection rate for 10.000 repetitions is shown in Table~\ref{tab:gof}. Such a $\chi^2$ test can also be used as a test for uniformity of the remainder $T^{n-1}(X)$. A more refined test can be performed by basing the goodness-of-fit test on $2^k$ combinations of the first $k$ digits $(X_n,\ldots,X_{n+k-1})$. The result is shown in Table~\ref{tab:gof} for $k=1,2,3$. When $n=1$ we always rejected the hypothesis that $(X_n,\ldots,X_{n+k-1})$ is uniformly distributed, when $n=2$ the rejection rate decreases as $k$ grows and it is 0.067 for $k=3$, and when $n\ge3$ the rejection rates are close to 0.05 as expected if the hypothesis is true. When we instead tried with a sample of 100 observations, even when $n=1$ the test had almost no power for $k=1,2,3$. 
	\begin{table}\label{tab:gof}
		\begin{center}
			\begin{tabular}{lcccccccc}
				\hline 
				$n$ &1&2&3&4&5&6&7&8
				\\
				\hline
				$k=1$ &1.000& 0.094&0.050&0.054&0.052&0.051&0.054&0.055
				\\
				$k=2$ &1.000&0.081&0.047&0.050& 0.051&0.053&0.052&0.047
				\\
				$k=3$ &1.000&0.067& 0.050& 0.049&0.049&0.050&0.052& 0.052
				\\
			\end{tabular}
		\end{center}
		\caption{Rejection rate for a $\chi^2$ goodness-of-fit  test for uniformity of $(X_n,\ldots,X_{n+k-1})$. }
	\end{table}
\end{example}

\begin{example}\label{ex:2}
	To illustrate how the convergence rate in Theorem \ref{t:1} depends on the smoothness of $f$,
	let $f(t)=\alpha t^{\alpha-1}$ be a beta-density with shape parameters $\alpha>0$ and $1$. Then, $f\in \bar L_1'(0,1)$ if and only if $\alpha=1$ or $\alpha\ge2$. Of course, $P_n$ and $\mu$ agree if $\alpha=1$. For $q=2$, Figure~\ref{fig:plots} shows plots of $d_{\mathrm{TV}}(P_n,\mu)$ and  $\ln(d_{\mathrm{TV}}(P_n,\mu))$ versus $n$ when $\alpha=0.1,0.5,1,1.5,5,10$ as well as  a plot of $\ln(\frac{1}{8} q^{-2n}  \sum_{j=0} ^{q^n-1}\|f'_{n,j}\|_\infty) - \ln(d_{\mathrm{TV}}(P_n,\mu) )$ (cf.\ \eqref{twicediff}) versus $n$ when $\alpha=2,5,10$. For the calculation of $d_{\mathrm{TV}}(P_n,\mu)$ 
	observe that $f_n'$ is $<0$ if $\alpha<1$ and $>0$ if $\alpha>1$, so 
	$f_n(x_0)=1$ for some unique $x_0\in(0,1)$, and hence since $F_n(0)-0=F_n(1)-1=0$,
	$$
	d_{\mathrm{TV}}(P_n,\mu) =\frac{1}{2}\|f_n-1\|_1= \frac{1}{2}\bigg| \int_0^{x_0} (f_n(t) - 1) \mathrm d t  \bigg| +  \frac{1}{2}\bigg| \int_{x_0}^{1} (f_n(t) - 1) \mathrm d t  \bigg| = |F_n(x_0)-x_0|.
	$$ 
	 We used the Newton-Raphson procedure to find $x_0$ (the procedure always converges). 
	
	The first plot in Figure~\ref{fig:plots} shows that for all values of $\alpha$, $d_{\mathrm{TV}}(P_n,\mu)$ goes to zero, as guaranteed by Theorem \ref{t:1}. The second plot indicates that for $\alpha>1$, $d_{\mathrm{TV}}(P_n,\mu)$ decays exponentially at a rate independent of $\alpha$, while for $\alpha<1$, the decay is also exponential, but with a slower rate. The graphs in the third plot seem to approach zero, indicating that for $\alpha\ge 2$, the rate of decay is indeed as given by \eqref{twicediff}, which holds since $f''$ is bounded. In the middle plot, the decay rate also seems to be $q^{-n}$ for $\alpha=1.5$, though this is not guaranteed by Theorem \ref{t:1}. 
	To see why the rate $q^{-n}$ also holds for $1<\alpha<2$, we argue as follows. In 
	\eqref{dens_diff}, \eqref{e:gn_diff}, and \eqref{e:dTVfg}, we may refine to the cases $j=0$ and $j>0$ (observing that $\|f'_{n,j}\|_\infty<\infty$ when $j>0$)
	 to obtain the following modification of \eqref{e:uniform3},
	\begin{equation*}
	d_{\mathrm{TV}}(P_n,\mu) \le q^{-n} \left(\frac{1}{2}\|f_{n,0}\|_\infty + \frac{1}{6}q^{-n}\sum_{j=1}^{q^n-1}\|f'_{n,j}\|_\infty\right).
	\end{equation*}	
	Furthermore, 
	since $|f'|$ is decreasing for $\alpha<2$, $\sum_{j=1}^{q^n-1} \|f_{n,j}'\|_\infty q^{-n}$ is a lower Riemann sum for the improper Riemann integral $\int_0^1 |f'(t)|\, \mathrm dt$, which exists and is finite when $1<\alpha <2$. Consequently, for every $x\in (0,1)$,
	\[
	d_{\mathrm{TV}}(P_n,\mu) \leq q^{-n}\left(\frac{1}{2}\|f_{n,0}\|_\infty + \frac{1}{6}\|f'\|_1\right).
	\]

	\begin{figure}\label{fig:plots}
		\includegraphics[width=\textwidth]{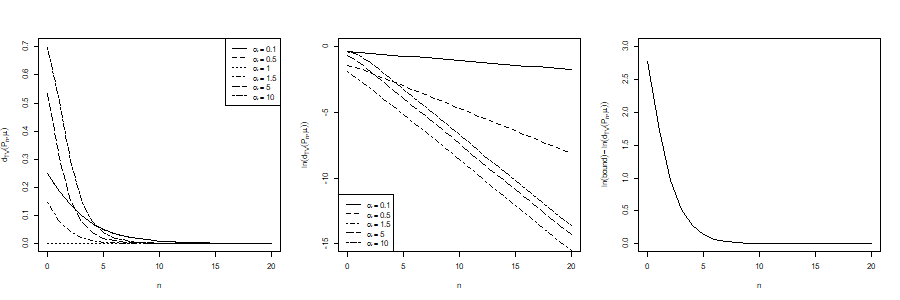}
		\caption{The first two plots show $d_{\mathrm{TV}}(P_n,\mu)$ and $\log(d_{\mathrm{TV}}(P_n,\mu))$, respectively, as a function of $n$ for $q=2$ and various values of $\alpha$. The last plot shows the difference between $\log(\frac{1}{8} q^{-2n}  \sum_{j=0} ^{q^n-1}\|f'_{n,j}\|_\infty)$ and $\log(d_{\mathrm{TV}}(P_n,\mu))$ for three values of $\alpha\ge2$.}
	\end{figure}
	
	As in Example \ref{ex:1}, we tested for uniformity of $T^{n-1}(X)$  by a $\chi^2$ goodness-of-fit test for uniform distribution of the $k=3$ first digits $(X_n,X_{n+1},X_{n+2})$ again using 10.000 replications of samples of 1000 observations from a beta-distri\-bu\-tion with $\alpha=0.1,0.5,1.5,2$. Table~\ref{tab:gof2} shows that for $\alpha=0.1$, uniformity is rejected in all samples for all $n$ indicating that the distribution of the remainder remains far from uniform even for $n=8$. For $\alpha= 0.5$, the rejection rate reaches the 0.05 level for $n=5$, while for $\alpha=1.5$, this happens already for $n=2$ and for $\alpha=5$ it happens around $n=3$ or $n=4$. For $\alpha>1$ close to 1, the results are comparable to those for the Benford law in Example~\ref{ex:1}, while for large $\alpha$ and $\alpha<1$, the rejection rate is higher indicating slower convergence. 
	
	\begin{table}\label{tab:gof2}
		\begin{center}
			\begin{tabular}{lcccccccc}
				\hline 
				$n$ &1&2&3&4&5&6&7&8
				\\
				\hline
				$\alpha=0.1$ & 1.000& 1.000&1.000&1.000&1.000&1.000&1.000&1.000
				\\
				$\alpha=0.5$ &1.000&1.000&0.416&0.078&0.052 &0.051& 0.048 &0.049
				\\
				$\alpha=1.5$ &1.000&0.123& 0.049&0.048 &0.051&0.052&0.047&0.051 
				\\
				$\alpha=5$ &1.000&0.932& 0.059& 0.049 & 0.048  & 0.048 &0.049& 0.050 
				\\
			\end{tabular}
		\end{center}
		\caption{Rejection rate for a $\chi^2$ goodness-of-fit  test for uniformity of $(X_n,\ldots,X_{n+2})$ in a beta-distribution for various values of $\alpha$. }
	\end{table}    
	
\end{example}

\br In conclusion, Examples~\ref{ex:1} and \ref{ex:2} demonstrate that the answer to the title of our paper (`How many digits are needed?') of course depend much on $q$ (in Example~\ref{ex:1}, the higher $q$ is, the fewer digits are needed) and on how much $f$ deviates from the uniform PDF on $[0,1)$ (in Example~\ref{ex:2}, the more skew $f$ is, the more digits are needed). Moreover, as this deviation increases or the sample size decreases, the $\chi^2$ goodness-of-fit test as used in the examples becomes less powerful; alternative tests are discussed in \cite{Morrow}.
\er 

\section{The multivariate case}\label{s:multi}

Theorem \ref{t:1} extends as follows. For a given positive integer $k$, 
let now $X=(X_1,...,X_k)$ be a $k$-dimensional random variable with values in the unit cube $[0,1)^k$ so that its CDF $F(x_1,...,x_n)=\mathrm P(X_1\le x_1,...,X_k\le x_k)$ is absolutely continuous, and denote its multivariate PDF by $f$. Extend the function $T$ to be a function $T: [0,1)^k\mapsto[0,1)^k$ so that $T(x_1,...,x_k)=(T(x_1),...,T(x_k))$.  For $n=1,2,...$, denote the multivariate CDF of $T^n(X)$ by $F_n$.
For a real Lebesgue integrable function $g$ defined on $(0,1)^k$, let $\|g\|_1=\int_{(0,1)^k}|g(t)|\,\mathrm dt$ 
and let $L_1((0,1)^k)$ be the set of such functions $g$ (i.e., $\|g\|_1<\infty$).
For a real $k$-dimensional function $g=(g_1,...,g_k)$ defined on $(0,1)^k$, let 
$\|g\|_\infty=\sup_{x\in(0,1)^k}\sqrt{g_1(x)^2+...+g_k(x)^2}$. Define the set $\bar L_1((0,1)^k)=\{g \in L^1((0,1)^k)\,|\, \int_{(0,1)^k}g(t)\,\mathrm dt =1\}$  and $\bar L_1'((0,1)^k)$ as its subset of differentiable functions $g$ with gradient 
\[\nabla g(x_1,\ldots,x_k)=\left(\frac{\partial g}{\partial x_1}(x_1,\ldots,x_k),\ldots,\frac{\partial g}{\partial x_k}(x_1,\ldots,x_k) \right)\quad \]
such that $\|\nabla g\|_{\infty}<\infty$. Thus, for $k=1$, $\bar L_1'((0,1)^k)=\bar L_1'(0,1)$ as used in Theorem~\ref{t:1}.
For $g\in \bar L_1'((0,1)^k)$, $n\in\mathbb N$, $j: = (j_1,...,j_k)\in\{0,1,...,q-1\}^k$, and $x=(x_1,...,x_k)\in(0,1)^k$, define  $\nabla g_{n,j}(x)=\nabla g(x)$ if $q^{-n}j_i<x_i<q^{-n}(j_i+1)$ for $i=1,...,k$ and $\nabla g_{n,j}(x)=0$ otherwise, 
and define
$$F_g(x) = \int_0^{x_1}\cdots\int_0^{x_k} g(t_1,...,t_k)\,\mathrm dt_1\cdots\,\mathrm dt_k.$$ 
For notational convenience, we can consider $F$ and $F_n$ to be functions defined on $(0,1)^k$, so $F=F_f$.
Let $e=(1,\cdots,1)$, that is $x$ with each component equal to $1$, and as a short hand notation write $\sum_{j=0}^{(q^n-1)e}...$ for $\sum_{j_1=0}^{q^n-1}\cdots\sum_{j_k=0}^{q^n-1}...$, and for a real function $g$ defined on $(0,1)^k$, $n=1,2,...$, and $x\in(0,1)^k$, let
\[ g_n(x) = q^{-nk}\sum_{j=0}^{(q^n-1)e} g(q^{-n}(j+x)).\]
Then, as in \eqref{e:5555}, we see that $F_n$ is absolutely continuous with PDF $f_n$.
Finally, let $P_n$ be the probability distribution with CDF $F_n$, $\mu$ Lebesgue measure on $[0,1)^k$, and 
$d_{\mathrm{TV}}(P_n,\mu)$ 
the total variation distance between these measure
(where \eqref{e:TV2} extends to the multivariate case with obvious modifications).


\bt \label{t:2} If $g \in \bar L_1'((0,1)^k)$ then
\begin{align}\label{e:uniform1k}
d_{\mathrm{TV}}(P_n,\mu)
&\le \frac12\|f-g\|_1+\frac{1}{2}{\sqrt{\frac{k}{3}}} q^{-n(k+1)} \sum_{j=0}^{(q^n-1)e}\|\nabla g_{n,j}\|_\infty\\
&\le \frac12\|f-g\|_1+ \frac{1}{2}{\sqrt{\frac{k}{3}}} q^{-n} \|\nabla g\|_\infty.
\end{align}
In particular, 
\begin{equation*}
\lim_{n\rightarrow\infty} d_{\mathrm{TV}}(P_n,\mu)
= 0.
\end{equation*}
Furthermore,  if $f \in \bar L_1'((0,1)^k)$ then $P_n$ converges exponentially fast:
\begin{equation*}
d_{\mathrm{TV}}(P_n,\mu)
\le \frac12\sqrt{\frac{k}{3}}q^{-n(k+1)} \sum_{j=0}^{(q^n-1)e}\|\nabla f_{n,j}\|_\infty
\le \frac{1}{2}{\sqrt{\frac{k}{3}}}  q^{-n}
\|\nabla f\|_\infty.
\end{equation*}
Finally, if $\|f\|_\infty < \infty $ and $f$ is continuous except for finitely many points, then 
	\begin{equation} \label{e:uniform2k}
	|f_n(x)-1| \to 0 \quad \mbox{uniformly for $x\in(0,1)^k$}.
	\end{equation}
\et

\begin{proof} 
	Let $x\in(0,1)^k$ and $g \in \bar L_1'((0,1)^k)$.
	As in (\ref{dens_diff}),
	\begin{align*} 
	g_n(x) - 1 &= \sum_{j=0} ^{(q^n-1)e}\int_{q^{-n}j_1}^{q^{-n(j_1+1)}}\cdots \int_{q^{-n}j_k}^{q^{-n(j_k+1)}} [g(q^{-n}(j+x)) - g(t)]\, \mathrm d t \\
	&= \sum_{j=0}^{(q^n-1)e} \int_{(0,1)^k} q^{-nk} [g(q^{-n}(j+x)) - g(q^{-n}(j+y))]\, \mathrm d y .
	\end{align*}
	By the mean value theorem, 
	$$ |g(q^{-n}(j+x)) - g(q^{-n}(j+y))| \le \|\nabla g_{n,j}\|_\infty q^{-n}\|x-y\|$$
	where $\|\cdot\|$ is usual Euclidean distance.
	We estimate 
	$$\int_{(0,1)^k}\|x-y\|\,\mathrm dy \le \left(\int_{(0,1)^k}\|x-y\|^2\,\mathrm dy\right)^{1/2} \le \frac{\sqrt{k}}{3}$$ 
	which yields the bound
	$$|g_n(x) - 1| \leq q^{-n(k+1)}\frac{\sqrt{k}}{3}\sum_{j=0}^{(q^n-1)e}\|\nabla g_{n,j}\|_\infty \le q^{-n} \frac{\sqrt{k}}{3} \|\nabla g\|_\infty.$$
	As in (\ref{fn-gn}) we have
	$$ \int_0^1 \left|f_n(x) - g_n(x)\right|\,\mathrm dx \le \|f-g\|_1.$$  
	Combining the last two estimates gives
	\begin{align*}
	2d_{\mathrm{TV}}(P_n,\mu) 
	&\le \|f-g\|_1 + \sqrt{\frac{k}{3}}q^{-n(k+1)} \sum_{j=0}^{(q^n-1)e}\|\nabla g_{n,j}\|_\infty \\
	&\le \|f-g\|_1 + \sqrt{\frac{k}{3}}q^{-n}\|\nabla g\|_\infty 
	\end{align*}
	whereby \eqref{e:uniform1k} follows.
	
	To show \eqref{e:uniform2k}, we assume for simplicity that $f$ only has one discontinuity at $x_0=(x_{0,1},\ldots,x_{0,k}) \in (0,1)^k$. The case of more than one discontinuity can be treated as in the proof of Theorem \ref{t:1}. Let $\delta>0$ and define
		\begin{align*}
		I_n&=\{(j_1,\ldots,j_k)\in \{0,1,\ldots,q^n-1\}^k\mid \exists i: j_i< \lfloor q^n\delta \rfloor  \vee j_i> \lceil q^n(1-\delta) \rceil \},\\
		J_n& =\{(j_1,\ldots,j_k)\in \{0,1,\ldots,q^n-1\}^k\mid \max_i |j_i-q^nx_{0,i}|<\lfloor q^n\delta \rfloor \}, \\
		K_n&=\{0,\ldots,q^n-1\}^k\backslash (I_n\cup J_n). 
		\end{align*}
		Then
		\begin{align} \nonumber
		f_n(x) - 1 &= \sum_{j\in I_n} \int_{(0,1)^k} q^{-nk} [f(q^{-n}(j+x)) - f(q^{-n}(j+y))]\, \mathrm d y \\ \nonumber
		&+ \sum_{j\in J_n} \int_{(0,1)^k} q^{-nk} [f(q^{-n}(j+x)) - f(q^{-n}(j+y))] \, \mathrm d y\\ \label{e:3terms}
		&+ \sum_{j\in K_n} \int_{(0,1)^k} q^{-nk} [f(q^{-n}(j+x)) - f(q^{-n}(j+y))]\, \mathrm d y.
		\end{align}
		If $\varepsilon>0$ is given, we can choose $\delta$ such that each term in \eqref{e:3terms} is less than $\varepsilon/3$. This follows  as in the proof of Theorem \ref{t:1} by noting that the cardinality of $I_n$ is at most $\delta q^{n(k-1)}$, the cardinality of $J_n$ is at most $  2\delta q^{n}$, and $f$ is bounded on $(0,1)^k$ and uniformly continuous on the closed set $[\delta/2,1-\delta/2]^k \backslash C(x_0)$ where $C(x_0)$ denotes the cube of sidelength $\delta$ centered at $x_0$.
\end{proof}

\section*{Acknowledgements} 
\noindent 
Supported by The Danish Council for Independent Research —
Natural Sciences, grant DFF – 10.46540/2032-00005B.


\begin{thebibliography}{99}


	
	\bibitem{B-N} 
	{\sc Barndorff-Nielsen, O.E.}\ (1978). {\em Information and Exponential Families in Statistical Theory.} Wiley, Chichester.
	
	\bibitem{Berger} 
	{\sc Berger, A. and Hill, T.P.}\ (2015). {\em An Introduction to Benford’s Law.} Princeton University Press, Princeton, NJ.
	
	
	\bibitem{Borel} 
	{\sc Borel, E.}\ (1909). Les probabilit{\'e}s d{\'e}nombrables et leurs applications arithm{\'e}tiques. {\em Rendiconti del Circolo Matematico di Palermo} {\bf 27}, 247--271.
	
	\bibitem{part1}
	{\sc Cornean, H., Herbst, I., M{\o}ller, J., St{\o}ttrup, B.B., Studsgaard, K.S.}\ (2022). Characterization of random variables with stationary digits. {\em J.\ Appl.\ Prob.} {\bf 59}, 931--947.
	
	\bibitem{Karma}
	{\sc Dajani, K., Kalle, C.}\ (2021).
	{\em A First Course in Ergodic Theory.} Chapman and Hall/CRC.
	
	\bibitem{GibbsSu} {\sc Gibbs, A.L.,  Su, F.E.}\ (2002). On choosing and bounding probability metrics. {\em Int. Stat. Rev.}, {\bf 70}, 419--435.
	
	
	\bibitem{Hill}
	{\sc Hill, T.P.}\ (1998). The first digit phenomenon. {\em American Scientist.}  {\bf 86}, 358--363.
	
	\bibitem{Morrow}
	{\sc Morrow, J.}\ (2014). {\em Benford's Law, families of distributions and a test basis.} CEP Discussion Paper No 1291, Centre for Economic Performance, London School of Economics and Political Science.
	
	\bibitem{Fritz}
	{\sc Schweiger, F.}\ (1995). {\em Ergodic Theory of Fibered Systems and Metric Number Theory.} Clarendon Press, Oxford.
	
	
	\bibitem{thorisson}
	{\sc Thorisson, H.}\ (2000) {\em  Coupling, Stationarity, and Regeneration.} Springer, New York.
	
	\bibitem{Tysbakov}
	{\sc Tysbakov, A.}\ (2009).  {\em Introduction to Nonparametric Estimation.} Springer Series in Statistics, Springer - Verlag, New York.
	
	
\end{thebibliography}
\end{document}